\documentclass[11pt,oneside]{amsart}
\usepackage[utf8]{inputenc}

\usepackage{amsmath, comment}
\usepackage{amsthm, amsfonts, amssymb,latexsym, color, tikz}
\title[Pinned angles determined by Cartesian products]{A lower bound for the number of pinned angles determined by a Cartesian product set}
\author{}
\usepackage{amsmath}
\usepackage[a4paper, total={6in, 8in}]{geometry}
\numberwithin{exercise}{subsection}
\newtheorem{lemma}{Lemma}

\newtheorem{theorem}{Theorem}
\newtheorem*{theorem*}{Theorem}

\author[O. Roche-Newton]{Oliver Roche-Newton} \address{Institute for Algebra, Johannes Kepler Universit\"{a}t\\
Linz, Austria}
\email{o.rochenewton@gmail.com}

\begin{document}

\begin{abstract}

We prove that, for any $B \subset \mathbb R$, the Cartesian product set $B \times B$ determines $\Omega(|B|^{2+c})$ distinct angles.
\end{abstract}

\maketitle

\section{Introduction}

Given three distinct points $p,q,r \in \mathbb R^2$, let $\mathcal A(p,q,r)$ denote the angle determined by the three points with $q$ as the centre. Given a set $P \subset \mathbb R^2$ with cardinality $|P|=n$, define
\[
\mathcal A(P):= \{ \mathcal A(p,q,r): p,q,r  \in P\}.
 \]
This paper is concerned with lower bounds for the size of the set $\mathcal A(P)$ as a function of $|P|$. As a starting point, observe that the bound
\begin{equation} \label{linear}
|\mathcal A(P)| \gg |P|
\end{equation}
holds for any set $P$ which is not contained in a line.\footnote{Throughout this note, the notation
 $X\gg Y$ and $Y \ll X,$ are equivalent and mean that $X\geq cY$ for some absolute constant $c>0$. In addition, we use the symbols $\gtrsim$ and $\lesssim$ to supress logarithmic factors. To be precise  $X\gtrsim Y$ and $Y \lesssim X,$ are equivalent and mean that $X\geq cY(\log Y)^{-c'}$ for some absolute constants $c,c'>0$.}  One method for proving \eqref{linear} is to use Beck's Theorem. If there are $\Omega(|P|)$ points from $P$ on a line $\ell$, we may fix any point $q \in  P \setminus \ell$ and a point $p \in P \cap \ell$, and observe that the angles
\[
\mathcal A(p,q,r), \,\, r \in P \cap \ell
\]
are distinct. Otherwise, it follows from Beck's Theorem that there is a point $q \in P$ which determines $\Omega(|P|)$ directions with the other elements of $P$, and this immediately implies that $q$ determines $\Omega(|P|)$ angles with these points, using $q$ as the centre of the angles.

The bound \eqref{linear} cannot be improved in general. For instance, one may consider a point set $P$ consisting of $n-1$ points distributed equally on a circle, along with the centre of the circle. A similar example can be constructed whereby $n-2$ points are on a single line, and the two additional points are arranged to be symmetric with respect to the line such that the directions to the points on the line form an arithmetic progression.

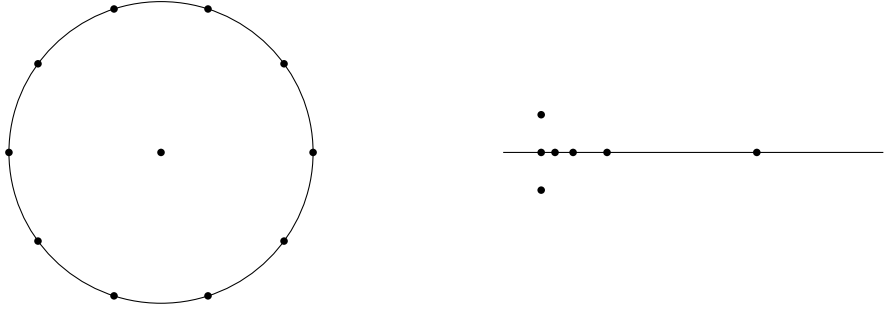
\begin{figure}

\begin{tikzpicture}
    \def\n{10}
    \def\r{2}
    \draw (0,0) circle (\r);
    \node[circle,fill,inner sep=1pt] at (0,0) {};
    \foreach \s in {1,...,\n}{
        \pgfmathsetmacro\angle{360/\n * (\s - 1)}
        \coordinate (P\s) at (\angle:\r);
        \node[circle,fill,inner sep=1pt] at (P\s) {};
    }

  \begin{scope}[shift={(5cm,0)}, scale=0.5]  \draw (-1,0) -- (9,0);
    \node[circle,fill,inner sep=1pt] at (0,1) {};
    \node[circle,fill,inner sep=1pt] at (0,-1) {};
    \node[circle,fill,inner sep=1pt] at (0,0) {};
   % \node[circle,fill,inner sep=1pt] at (0.1763,0) {};
    \node[circle,fill,inner sep=1pt] at (0.3639,0) {};
    %\node[circle,fill,inner sep=1pt] at (0.5773,0) {};
    \node[circle,fill,inner sep=1pt] at (0.8390,0) {};
    \node[circle,fill,inner sep=1pt] at (1.7320,0) {};
      \node[circle,fill,inner sep=1pt] at (    5.6712,0) {};

      \end{scope}

\end{tikzpicture}

\caption{This picture shows two degenerate examples sets of $n$ points determining $O(n)$ distinct angles. Essentially all of the examples that we are aware of for points sets with this property are derived from variations of these two configurations.}

\end{figure}

%\orn{Expand the forthcoming paragraph to discuss connections with the inverse problem for Erdos distances. Split into 2 paragraphs.}

We would like to prove that, if we impose some additional conditions on $P$ to avoid these special configurations, then $|\mathcal A(P)| \gg |P|^{1+c}$ holds for some absolute constant $c>0$. To put this another way, we would ultimately like to classify the points sets which determine few angles. For context, let us consider the analogue of this question for the Erd\H{o}s distance problem. Let $\mathcal D(P):= \{ \| p - q\| :p,q \in P \}$ denote the set of distinct distances determined by a point set $P \subset \mathbb R^2$. A breakthrough work of Guth and Katz \cite{GK} established that the near-optimal bound
\[
| \mathcal D(P) | \gg \frac{|P| }{ \log |P|}
\]
holds for all $P \subset \mathbb R^2$. The problem of determining the structure of sets which have close to the minimum possible number of distinct distances has attracted considerable attention. It is expected that such sets have a lattice-type structure, and it was conjectured by Erd\H{o}s \cite{E} that an extremal example must contain $\Omega( \sqrt{|P|})$ points on a line, although very little is currently known for this family of problems. For the analogous question concerning the structure of sets which determine few congruence classes of triangles, see a recent paper of Mansfield and Passant \cite{MP}.

Some work on super linear bounds for distinct angles was carried out in a very recent work of Konyagin, Passant and Rudnev \cite{KPR}, in which they proved that the bound
\[
|\mathcal A(P)| \gtrsim  |P|^{1+\frac{3}{23}}
\]
holds under the condition that $P$ is in convex position and not all of the points lie on a circle. Going back a little further, Pach and Sharir \cite{PS} proved an optimal bound for the maximum number of representations of a single angle. It is conceivable, bearing in mind the two examples illustrated above, that a condition that no $n-2$ points of $P$ lie on a line or circle is enough to guarantee a better than linear bound for $|\mathcal A(P)|$. However, to the best of our knowledge, it is not even known that such a bound holds under the much stronger assumption that the point set $P$ is in general position (that is, no $3$ points lie on a line and no $4$ points lie on a circle).

In this paper, we consider something of an intermediate case raised in \cite{KPR}, where the point set $P=B \times B$ is a Cartesian product. In this case, we are able to improve the trivial estimate \eqref{linear}, as follows.

\begin{theorem} \label{thm:main}
 For any finite set $B \subset \mathbb R^2$,
\[
|\mathcal A(B \times B)| \gg |B|^{2+\frac{1}{14}}.
\]
Moreover, there exists $q \in B \times B$ such that
\[
|\{ \mathcal A(p,q,r) : p,r \in B \times B \}| \gg |B|^{2+\frac{1}{14}}.
\]
\end{theorem}

In terms of the techniques at our disposal, the advantage of restricting to the case when $P$ is a Cartesian product is that it allows us to convert this geometric problem into an arithmetic question in the spirit of the sum-product problem. The problem reduces to a question about growth of a given set under a combination of additive and convex operations. Such questions have been studied using incidence geometry, going back to the work of Elekes, Nathanson and Ruzsa \cite{ENR}. Recently, an elementary ``squeezing'' argument has led to new progress in this direction; see, for example, \cite{RSSS}, \cite{HRNR} and \cite{B}. Of particular relevance here is the work of \cite{HRNS}, where these techniques were used to prove superquadratic growth estimates for convex expanders. The proof of Theorem \ref{thm:main} is also built upon applications of the squeezing argument.

\section{Preliminary results}

In the proof of Theorem \ref{thm:main}, we will use the following bipartite variant of Beck's Theorem. The result can be proved by a simple adaptation of the well-known proof of Beck's Theorem via an application of the Szemer\'{e}di-Trotter Theorem. For the convenience of the reader, a full proof is given below. The statement can also be derived as an immediate corollary of Theorem 1.9 from a recent paper of Lund, Pham and Thu \cite{LPT}.

\begin{theorem} \label{thm:beck}
There exists a sufficiently small absolute constant $c>0$ such that the following statement holds. Let $P$ and $Q$ be disjoint sets of $N$ points in $\mathbb R^2$ and suppose that at most $cN$ points from $Q$ are collinear. Then there exists $p=(p_1,p_2) \in P$ such that the set of directions
\[
\left \{ \frac{p_2-q_2}{p_1-q_1} : (q_1,q_2) \in Q \right \}
\]
has cardinality at least $cN$.
\end{theorem}

\begin{proof}
    Given $S \subset \mathbb R^2$, and an integer $t \geq 2$, let $L_t(S)$ denote the set of lines
    \[
    L_t(S): \{ \ell:  |\ell \cap S| \geq t \}.
    \]
    It follows from the Szemer\'{e}di-Trotter Theorem that
    \begin{equation} \label{STcor}
    |L_t(S)| \ll \frac{|S|^2}{t^3}+ \frac{|S|}{t}.
    \end{equation}

    Now, for two distinct points $p \in P$ and $q \in Q$, let $\ell_{p,q}$ denote the line connecting $p$ and $q$. We say that $p$ and $q$ are \textit{$k$-connected} if
    \[
    2^k \leq |\ell_{p,q} \cap Q| < 2^{k+1}.
    \]
    Note that that every pair $(p,q) \in P \times Q$ is $k$-connected for exactly one integer $k \geq 0$, and so
    \[
    N^2= \sum_{k \geq 0} | \{ (p,q) \in P \times Q:\text{ $p$ and $q$ are $k$-connected}\}|.
    \]
    We will prove the theorem with the absolute constant $c$ defined to be
    \begin{equation} \label{Cdef}
    c:=\frac{1}{33C'},
    \end{equation}
    where $C'$ is another absolute constant arising from applications of the Szemer\'{e}di-Trotter Theorem.
    
  %  Let $k_0$ be a sufficiently large absolute constant to be specified later. Consider the case when at least $\frac{N^2}{2}$ of the pairs in $P \times Q$ are $k$-connected for some $k \leq k_0$. It follows from the pigeonhole principle that there is some $p \in P$ such that
   % \[
   % |\{ q \in Q : \text{$p$ and $q$ are $k$-connected for some $k \leq k_0$} \}| \geq \frac{N}{2}.
   % \]
   % It then follows that the number of directions between $p$ and $Q$ is at least
   % \[
   % \frac{N}{2^{k_0+1}} \geq cN,
 %   \] 
   % giving the required result. In the last inequality above, we have imposed our first condition on the absolute constant $c$, and so we will need to choose $c$ at the end of the proof to ensure that
  %  \begin{equation} \label{ccond1}
  %  c \leq \frac{1}{2^{k_0+1}}.
  %  \end{equation}
    
    Suppose for a contradiction that at most $\frac{N^2}{2}$ of the pairs in $P \times Q$ are $k$-connected with $k \leq k_0$ for some sufficiently large absolute constant $k_0$ to be specified later. In particular, and using also the assumption that no more than $cN$ points of $Q$ lie on any line, we have
    \begin{align*}
    \frac{N^2}{2} &\leq \sum_{k = k_0+1}^{\lceil \log_2(cN) \rceil+1} | \{ (p,q) \in P \times Q:\text{ $p$ and $q$ are $k$-connected}\}|
    \\& = \sum_{k = k_0+ 1}^{\lceil \log_2(cN) \rceil+1} \sum_{ \,\, \ell : 2^k \leq | \ell \cap Q|<2^{k+1}} |\ell \cap P| \cdot | \ell \cap Q|
    \\& < \sum_{k = k_0+ 1}^{\lceil \log_2(cN) \rceil+1} 2^{k+1} \sum_{ \,\, \ell : 2^k \leq | \ell \cap Q|<2^{k+1}} |\ell \cap P| 
    \\& \leq   \sum_{k = k_0+ 1}^{\lceil \log_2(cN) \rceil+1} 2^{k+1} \cdot  I(P, L_{2^k}(Q))
    \\& \leq C \sum_{k = k_0+ 1}^{\lceil \log_2(cN) \rceil+1} 2^k\left ( |P|^{2/3}|L_{2^k}(Q)|^{2/3} +|P| + |L_{2^k}(Q)| \right )
    \\& \leq C'\sum_{k = k_0+ 1}^{\lceil \log_2(cN) \rceil+1} \left ( \frac{N^2}{2^k} + 2^k N   \right )
    \\& \leq C' \left ( \frac{N^2}{2^{k_0}} + 8cN^2 \right ),
    \end{align*}
    where $C$ and $C'$ are absolute constants coming from applications of the Szemer\'{e}di-Trotter Theorem, and $C'$ is the constant which appeared in the definition of $c$ back in \eqref{Cdef}.

    By choosing $k_0$ to be sufficiently large (concretely, we can set $k_0 := \lceil \log_2(4C') \rceil $), it follows that the first term on the right hand side can be absorbed into the left hand side to give
    \[
    \frac{N^2}{4} \leq 8C'cN^2 =\frac{8}{33}N^2,
    \]
    obtaining the intended contradiction.

    Therefore, we may restrict our attention to the case when at least $\frac{N^2}{2}$ of the pairs in $P \times Q$ are $k$-connected for some $k \leq k_0$. It follows from the pigeonhole principle that there is some $p \in P$ such that
    \[
    |\{ q \in Q : \text{$p$ and $q$ are $k$-connected for some $k \leq k_0$} \}| \geq \frac{N}{2}.
    \]
    It then follows that the number of directions between $p$ and $Q$ is at least
    \[
    \frac{N}{2^{k_0+2}} \geq \frac{N}{32C'} > cN.
    \] 
\end{proof}

The key ingredient for the proof of Theorem \ref{thm:main} is the following new result. %Some work is still required to figure out the details of the proof. Hopefully I can get this part sorted soon. I am confident that the method used in \cite{HRNS} will work.
\begin{theorem} \label{thm:key}
For any finite sets $X,Y, \subset \mathbb R$
\[
|4f(X-Y) - 3f(X-Y)| \gtrsim \min \{ |X|,|Y| \}^{5/2},
\]
where the function $f$ is given by the formula $f(x)=\arctan(e^x)$.
\end{theorem}

Theorem \ref{thm:key} can be derived from Theorem 2.6 in \cite{HRNS}. A special case of this result which is tailored slightly to our purposes is stated below.

\begin{theorem}\label{thm:MainExpander}
Let $f: \mathbb R\to\mathbb R$ be a continuous, strictly increasing and strictly convex function, and let $h_3>h_2>h_1$ be real numbers. Suppose that the curve
\begin{equation} \label{impy}
 \{(f(t+h_2)-f(t+h_1),f(t+h_3)-f(t+h_2)) : t\in \mathbb R\}
\end{equation}
is the graph of a strictly convex or strictly concave function.  Suppose that $A=\{a_1<\ldots<a_N\}$ is a finite set of positive real numbers satisfying the spacing condition
\[
\min\{a_{i+1}-a_i:i\in[N-1]\}\geq \max\{h_2-h_1, h_3-h_2\}.
\] 
Then
\begin{multline*}|2f(A+h_1)-2f(A+h_1)+2f(A+h_2)-f(A+h_2)|\cdot\\\cdot|2f(A+h_3)-2f(A+h_3)+2f(A+h_2)-f(A+h_2)| \gg \frac{N^5}{(\log N)^3}.
\end{multline*}
\end{theorem}

To apply Theorem \ref{thm:MainExpander} with the function $f(x)= \arctan(e^x)$, we need to verify the technical condition that the parametrically defined curve \eqref{impy} is concave. This is the content of the following lemma.

\begin{lemma}
Let $f(x)=\arctan(e^x)$ and suppose that $h_1<h_2<h_3$ are real numbers. The curve
\[
 \{(f(t+h_2)-f(t+h_1),f(t+h_3)-f(t+h_2)) : t\in \mathbb R\}
 \]
 is the graph of a strictly concave function.
\end{lemma}

\begin{proof}
A direct and slightly laborious calculation gives
\[
 \frac{d^2y}{dx^2}=\frac{2(1+e^{2(t+h_3)})(1+e^{2(t+h_1)})(e^{t+h_3}-e^{t+h_2})(e^{2(t+h_1)}-e^{2(t+h_3)})}{(e^{t+h_2}-e^{t+h_1})(1+e^{2(t+h_3)})^2(e^{t+h_2}-e^{t+h_1})}.
 \]
 This is strictly negative for all $t$. Indeed, all of the factors are strictly positive for all $t$, with the exception of
 \[
 e^{2(t+h_1)}-e^{2(t+h_3)},
 \]
 which is always negative. Since the second derivative is strictly negative, it follows that the curve is the graph of a strictly concave function.

\end{proof}

\begin{proof}[Proof of Theorem \ref{thm:key}]

%We will apply Theorem \ref{thm:MainExpander} with $f(x)= \arctan(e^x)$. For this purpose, we will need to prove the following claim.

Label the elements of $X=\{x_1 < \dots < x_m \}$ and $Y=\{y_1 < \dots < y_n \}$ in ascending order. Consider the set of all nearly neighbouring distances
\[
U=\{x_{k+2} - x_k:  1 \leq k \leq m-2 \} \cup \{y_{k+2} - y_k:  1 \leq k \leq n-2 \}
\]
and let $u$ denote the minimal element of $U$. 

\textbf{Case 1} - Suppose that $u=x_{k+2} - x_k$ for some $1 \leq k \leq n-2$.  
Let 
\[
A=\left \{ -y_{2k} : 1 \leq k \leq \frac{n}{2} \right \}.
\]
Apply Theorem \ref{thm:MainExpander} with this $A$ and with
\[
h_1=x_k, \, h_2=x_{k+1}, \, h_3=x_{k+2}.
\]
It follows that
\[
|4f(X-Y)-3f(X-Y)| \gtrsim |Y|^{5/2}.
\]

\textbf{Case 2} - Suppose that $u=y_{k+2} - y_k$ for some $1 \leq k \leq m-2$.  
Let 
\[
A=\left \{ x_{2k} : 1 \leq k \leq \frac{m}{2} \right \}.
\]
Apply Theorem \ref{thm:MainExpander} with this $A$ and with
\[
h_1=-x_{k+2}, \, h_2=-x_{k+1}, \, h_3=-x_{k}.
\]
It follows that
\[
|4f(X-Y)-3f(X-Y)| \gtrsim |X|^{5/2},
\]
as required. \end{proof}

\begin{comment}
Here is an attempt at a proof. Label the elements of $X=\{x_1 < \dots < x_n \}$ and $Y=\{x_1 < \dots < y_m \}$ in ascending order. Let us assume without loss of generality that $n \geq m$.

Consider the set of all nearly neighbouring distances
\[
U=\{x_{k+2} - x_k:  1 \leq k \leq n-2 \} \cup \{y_{k+2} - y_k:  1 \leq k \leq y-2 \}
\]
and let $u$ denote the minimal element of $U$. We assume for now that $u=x_{k+2} - x_k$ for some $1 \leq k \leq n-2$. The case when the minimal element of $U$ comes from a difference in $Y$ can be handles similarly, and we will make some further remarks about this at the end of the proof.

For convenience, we label
\[
x:= x_k, \, x'=x_{k+1}, \, x''=x_{k+2}.
\]
These three elements are fixed for the remainder of the proof.

Consider the intervals
\[
(x-y_{2i}, x'-y_{2i}), \,\, 1 \leq i \leq \lfloor m/2 \rfloor. 
\]
These all have length $x'-x$. Also, they are pairwise disjoint. Indeed, suppose for a contradiction that two such intervals overlap.
\end{comment}

\section{Proof of Theorem \ref{thm:main}}

We need to take care that a forthcoming application of the logarithmic function is applied only to postitive values. For this reason, we split the set $B$ into two disjoint sets $B_1$ and $B_2$ such that
\[
|B_1|=|B_2| \geq \left \lfloor \frac{|B|}{2}  \right \rfloor
\]
and with the property that the largest element of $B_1$ is smaller than than smallest element of $B_2$.

Write $P= B_1 \times B_1$ and $Q=B_2 \times B_2$.  First, it follows from Theorem \ref{thm:beck} that there exists $(a,b) \in P$ such that $(a,b)$ determines $\Omega(|B|^2)$ directions with the set $Q$. That is,
\begin{equation} \label{quad}
\left | \frac{B_2-b}{B_2-a} \right | \gg |B|^2.
\end{equation}
We can calculate that
\[
\mathcal A( (x,y),(a,b), (x',y'))= \arctan\left ( \frac{y-b}{x-a} \right) - \arctan\left ( \frac{y'-b}{x'-a} \right).
\]
Define
\[
\mathcal A(Q,(a,b),Q)= \{ \mathcal A (p,(a,b),r) : p, r \in Q \}.
\]
Then
\begin{equation} \label{defn}
|\mathcal A(B \times B)| \geq |\mathcal A(Q,(a,b),Q)|= \left | \arctan\left ( \frac{B_2-b}{B_2-a} \right) - \arctan\left ( \frac{B_2-b}{B_2-a} \right) \right |.
\end{equation}
Let $X= \log (B_2-b)$ and $Y= \log(B_2-a)$. Note that these sets are well-defined, since the sets $B_2-a$ and $B_2-b$ consist of strictly positive elements. Applying Theorem \ref{thm:key} with these input sets gives
\[
 \left |4 \arctan\left ( \frac{B_2-b}{B_2-a} \right) - 3\arctan\left ( \frac{B_2-b}{B_2-a} \right) \right | \gtrsim |B|^{5/2}.
\]
Plünnecke's Theorem then gives
\begin{align*}
|B|^{5/2} &  \lesssim  \left |4 \arctan\left ( \frac{B_2-b}{B_2-a} \right) - 3\arctan\left ( \frac{B_2-b}{B_2-a} \right) \right |
\\& \leq \frac{ \left | \arctan\left ( \frac{B_2-b}{B_2-a} \right) -\arctan\left ( \frac{B_2-b}{B_2-a} \right) \right |^7}{ \left | \frac{B_2-b}{B_2-a} \right |^6}.
\end{align*}
Applying \eqref{quad} and recalling \eqref{defn} completes the proof.

\subsection*{Acknowledgements} The author was supported by the Austrian Science Fund FWF Projects P 34180 and PAT 2559123. I am grateful to Krishnendu Bhowmick, Misha Rudnev and Audie Warren for helpfully sharing their insights.

\end{document}